%% file: KLfoulkeshowe10-27.tex
\documentclass[11pt]{amsart}
\usepackage{latexsym,amssymb,amsmath,youngtab}
\usepackage{youngtab}
\usepackage{epsfig}
\usepackage{amsmath,amsthm,amssymb,amscd,MnSymbol}
\usepackage[mathscr]{eucal}
\usepackage{verbatim}
\usepackage{color}
\input cortdefs.tex

\def\tmult{{\rm mult}}
\newcommand{\GL}{\operatorname{GL}}
\newcommand{\SL}{\operatorname{SL}}
\newcommand{\SU}{\operatorname{SU}}
\newcommand{\Ch}{\operatorname{Ch}}
\newcommand{\Ind}{\operatorname{Ind}}
\newcommand{\triv}{\operatorname{triv}}
\newcommand{\Mat}{\operatorname{Mat}}
\newcommand{\Sym}{\operatorname{Sym}}

\begin{document}

\title[Latin squares and representation theory] {Connections between conjectures of Alon-Tarsi, Hadamard-Howe,
and integrals over the special unitary group}
\author{Shrawan Kumar and J.M. Landsberg}
 \begin{abstract} 
We show the Alon-Tarsi conjecture on Latin squares
is equivalent to a very special case of a conjecture
made independently by  Hadamard and Howe, and
to the non-vanishing of some interesting integrals over $\SU(n)$.
Our investigations were motivated by geometric complexity theory. 
\end{abstract}
\thanks{Kumar was supported by the NSF grant number DMS-1201310 and Landsberg was supported by the NSF grant  DMS-1006353}
\email{shrawan@email.unc.edu, jml@math.tamu.edu}
\maketitle

\section{Introduction}
We first describe the conjectures of Alon-Tarsi, Hadamard-Howe, integrals over the 
special unitary
group, and  a related conjecture of Foulkes. We then state the equivalences
(Theorem \ref{kconjequiv})
and prove them.

\subsection{Combinatorics I: The Alon-Tarsi conjecture} 
Call an $n\times n$ array of natural numbers  a {\it Latin square} if each row
and column consists of  $[n]:=\{1\hd n\}$. 
Each row and column of a Latin square defines a permutation  $\s$ of $n$, where
the ordered entries of the row (or column) are
$\s(1)\hd \s(n)$. Define the sign of the row/column to be the sign of this
permutation. Define the {\it column sign} of the Latin square to
be the product of all the column signs (which is $1$ or $-1$, respectively
called {\it column even} or {\it column odd}), the {\it row sign} of the  Latin square
to be the product of the row signs and the {\it sign} of the Latin
square to be the product of the row sign and the column sign.

\begin{conjecture}\label{ATconj} \cite{MR1179249} [Alon-Tarsi] Let $n$ be even, then
the number of even Latin squares of size $n$ does not equal the
number of odd Latin squares of size $n$.
\end{conjecture}

Conjecture \ref{ATconj} is known to be true when $n=p\pm 1$,  where $p$ is
an odd prime; in particular, it is known to be true up to $n=24$ \cite{MR2646093,MR1451417}.

The Alon-Tarsi conjecture is known to be equivalent to several other
conjectures in combinatorics. For our purposes,  the most important is
the following due to Huang and Rota:

\begin{conjecture}\label{HRconj} \cite{MR1271866}
[Column-sign Latin square conjecture] Let $n$ be even, then
the number of column even Latin squares of size $n$ does not equal the
number of column odd Latin squares of size $n$.
\end{conjecture}

\begin{theorem} \label{AT=HR} \cite[ Identities 8,9]{MR1271866} The difference between the number
of column even Latin squares of size $n$ and the number of column odd Latin
squares of size $n$ equals the difference between the number of even Latin
squares of size $n$ and the number of odd Latin squares of
size $n$,  up to sign. In particular, the Alon-Tarsi conjecture holds for $n$ if and only if 
the column-sign Latin square conjecture holds for $n$.
\end{theorem}

\begin{remark} It is easy to see that for $n$ odd, the  number of even Latin squares of size $n$  equals the
number of odd Latin squares of size $n$.
\end{remark}

\subsection{The Hadamard-Howe conjecture}  
Let  $V$ be a finite dimensional    complex vector space, 
let $V^{\ot n}$ denote the space  of multi-linear maps
$V^*\ctimes V^*\ra\BC$,   the space of {\it tensors}. The permutation group
$\FS_n$ acts on $V^{\ot n}$ by permuting the inputs of the map.
Let $S^nV\subset V^{\ot n}$ denote the subspace of symmetric tensors, 
the tensors invariant under $\FS_n$,  which
we may also view as the space of homogeneous polynomials of degree $n$ on
$V^*$. Let $\Sym(V):=\oplus_d S^dV$, which is an algebra under
multiplication of polynomials. Let $\GL(V)$ denote the general linear group of invertible linear maps $V\ra V$.
Consider the $\GL(V)$-module map
$$h_{d,n}: S^d(S^nV)\ra S^n(S^dV)$$ given as follows:
   Include  $S^{d}(S^nV)\subset V^{\ot nd}$.
   Write $V^{\ot nd}=(V^{\ot n})^{\ot d}$, as $d$ groups of $n$ vectors reflecting the inclusion.
Now rewrite   $V^{\ot nd}=(V^{\ot d})^{\ot n}$ by grouping the first
vector space in each group of $n$ together, then the second vector space
in each group etc.. Next symmetrize within each group of $d$
 to obtain an element of   $(S^dV)^{\ot n}$, and finally symmetrize
 the groups to get an element of $S^n(S^dV)$.
 
 For example $h_{d,n}((x_1)^n\cdots (x_d)^n)=(x_1\cdots x_d)^n$
 and $h_{3,2}\bigl((x_1x_2)^3\bigr)=\frac{1}{4}x_1^3x_2^3+\frac{3}{4}(x_1^2x_2)(x_1x_2^2) $.

      The map $h_{d,n}$  was first considered
by Hermite \cite{hermite} who proved that, when  $\tdim V=2$,  the map is an isomorphism.
It had been conjectured by Hadamard \cite{MR1554881} and tentatively
conjectured  by Howe 
\cite{MR983608} (who wrote \lq\lq it is not unreasonable to expect\rq\rq ) that 
$h_{d,n}$ is always of maximal rank, i.e., injective for $d\leq n$ and surjective
for $d\geq n$. A consequence of
the theorem of  \cite{MR2172706} (explained below) is that, contrary
to the expectation above, $h_{5,5}$ is not an isomorphism.

For any $n\geq 1$, define the {\it Chow variety}
$$
\Ch_n(V^*):=\{ P\in S^nV^*\mid P=\ell_1\cdots \ell_n 
{\rm{\ for \ some \  } }\ell_j\in V^*\}.
$$
(This is a special case of a Chow variety, namely of the zero cycles in   projective
space, but it is the only one that we discuss in this article.)
In \cite{MR1243152,MR1601139}, Brion (and independently Weyman and
Zelevinsky) observed that
$\oplus_d S^n(S^dV)$ is the coordinate ring of the normalization of the 
Chow variety. (Given an irreducible  affine variety $Z$, its {\it normalization} $\tilde{Z}$ is 
an irreducible affine variety whose ring of regular functions is integrally closed and such that
there is a regular, finite, birational map $\tilde{Z}\ra Z$, see e.g., \cite[\S II.5]{MR1328833}.)
\begin{lemma}
(Hadamard, see  e.g. \cite[\S 8.6]{MR2865915}) \label{hadlem}
 The kernel of the  $\GL(V)$-module map
$$\oplus h_{d,n}:\Sym(S^nV):=\oplus_d S^d(S^n V)\ra \oplus_d S^n(S^dV)$$
is the 
ideal of the Chow variety. 
\end{lemma}

Brion also showed that for $d$ exponentially  large with respect to
$n$,  $h_{d,n}$ is surjective \cite{MR1601139}. McKay \cite{MR2394689} showed that if
$h_{d,n}$ is surjective, then $h_{d',n}$ is surjective for all $d'>d$, using
$h_{d,n:0}$ defined below. It is also known that if $h_{d,n}$ is surjective,
then $h_{n,d}$ is injective, see \cite{MR3169697}.

 The $\GL(V)$-modules appearing in the tensor algebra of $V$ are indexed by partitions
$\pi=(p_1\geq p_2\geq \dots \geq p_q\geq 0), \, q\leq \dim V$, and denoted $S_{\pi}V$. If $\pi$ is a partition of $d$, 
i.e., $p_1+\cdots +p_q=d$, the module $S_{\pi}V$ appears in $V^{\ot d}$ and in no other degree.
We will use the notation $s\pi:=(sp_1\hd sp_q)$.
 Repeated numbers in partitions are sometimes expressed as exponents when there is no danger of
confusion, e.g.,  $(3,3,1,1,1,1)=(3^2,1^4)$.
Let $\SL(V)$ be the subgroup of $\GL(V)$ consisting of determinant $1$ elements, and
let $\fsl(V)$ denote its Lie algebra.

This paper addresses
a very special case of the general problem of determining 
the $\GL(V)$-module $\tker h_{d,n}$: simply to determine
whether or not the module
$S_{(d^n)}V$ is in the kernel.

\begin{conjecture} \label{kumarconj} \cite{kumarcoordring}
For all $d$ and $n$, $S_{(d^n)}V$ is not in the kernel of $h_{d,n}: S^d(S^nV)\ra S^n(S^dV)$.
\end{conjecture}

\subsection{Combinatorics II: Foulke's conjecture}
The dimension of $V$, as long as it is at least $d$, is irrelevant
for the $\GL(V)$-module structure of the  kernel of $h_{d,n}$. 
{\it In this section we assume $\tdim V=dn$. }
Choose  a linear isomorphism $V\simeq \mathbb{C}^{nd}$. The
 Weyl group $\cW_V$ of $\GL(V)=\GL(nd)$, which can be thought of 
as the subgroup of $\GL(nd)$ consisting of the permutation matrices (in particular, it is isomorphic
to $\FS_{dn}$),  acts on  $V^{\ot dn}$ by acting on each factor. (We write $\cW_V$ to
distinguish this from the   $\FS_{dn}$-action permuting the
factors.)      
An element $x\in V^{\ot dn}$ has $\fsl(V)$-weight zero if, in the standard  basis 
$\{ e_i\}_{1\leq i\leq dn}$ of $V$ induced from the identification $V\simeq \mathbb{C}^{nd}$, $x$ is
a sum of monomials $x=\sum_{I=(i_1, \dots ,  i_{nd})}\,x^I e_{i_1}\otc e_{i_{nd}}$,
where $I$ runs over the orderings of $[nd]$.
If one restricts $h_{d,n}$ to the $\fsl(V)$-weight zero subspace, one
obtains a $\cW_V$-module map
$$h_{d,n: 0}: S^d(S^nV)_0\ra S^n(S^dV)_0.$$   
These $\cW_V$-modules are as follows:  Let $\FS_n\wr \FS_d\subset \FS_{dn}$
denote the wreath product, which, by definition,  is the normalizer of 
 $\FS_n^{\times d}$ in $\FS_{dn}$. It is the semi-direct product
 of $\FS_n^{\times d}$ with $\FS_d$, where $\FS_d$ acts by permuting the
 factors of $\FS_n^{\times d}$, see  e.g., \cite[p 158]{MR1354144}. 
Since $\tdim V=dn$, we get  $S^d(S^nV)_0=\Ind_{\FS_n\wr \FS_d}^{\FS_{dn}}\triv$,
where $\triv$ denotes the trivial $\FS_n\wr \FS_d$-module.
 
We obtain a  $\cW_V=\FS_{dn}$-module map   
$$
h_{d,n: 0}: \Ind_{\FS_n\wr \FS_d}^{\FS_{dn}} \triv \ra \Ind_{\FS_d\wr \FS_n}^{\FS_{dn}}\triv .
$$
Moreover, since 
every irreducible module appearing in $S^d(S^nV)$ has a non-zero
$\SL(V)$-weight zero subspace, $h_{d,n}$ is the unique  $\SL(V)$-module extension 
of $h_{d,n: 0}$.

The  map $h_{d,n:0}$  was defined purely in terms of combinatorics in \cite{MR977186} as   a path to try to prove
the following conjecture of Foulkes:

\begin{conjecture}  \cite{MR0037276}\label{foulkesconj}
  Let  $d>n$,    let $\pi$ be a partition
of $dn$ and let $[\pi]$ denote the corresponding $\FS_{dn}$-module. Then,  
$$
\tmult([\pi], \Ind_{\FS_n\wr \FS_d}^{\FS_{dn}}\triv)\geq 
\tmult([\pi], \Ind_{\FS_d\wr \FS_n}^{\FS_{dn}}\triv).
$$ 
\end{conjecture}

Conjecture \ref{foulkesconj}
  was shown to hold   asymptotically by L. Manivel in \cite{MR1651092},
in the sense that  
  for any partition $\mu$, the multiplicity of the partition 
$(dn-|\mu |, \mu)$ is the same in $S^d(S^nV)$ and $S^n(S^dV)$ as soon as
$d$ and $n$ are    at least $|\mu|$.
Conjecture \ref{foulkesconj}  is still open in general. However, the map $h_{5,5:0}$ was shown
not to be injective in \cite{MR2172706}, and thus $h_{5,5}$ is not injective.
The $\GL(V)$-module structure of the kernel of $h_{5,5}$ was determined
by C. Ikenmeyer and S. Mrktchyan as part of a 2012 AMS MRC program:
\begin{proposition}[Ikenmeyer and  Mkrtchyan, unpublished]  
  The kernel of $h_{5,5} :S^5(S^5\BC^5)\ra S^5(S^5\BC^5)$ consists of irreducible modules corresponding to the following partitions:
\begin{align*}\{
&(14,7,2,2), (13,7,2,2,1), (12,7,3,2,1), (12,6,3,2,2),\\
& (12,5,4,3,1), (11,5,4,4,1) ,(10,8,4,2,1) ,(9,7,6,3)\}.
\end{align*}
All these occur with multiplicity one in the kernel, but not all occur with multiplicity one in $S^5(S^5\BC^5)$. 
In particular, the kernel is not an isotypic component.
\end{proposition}

\subsection{Integration over $\SU(n)$}\label{integral}
Let $d\mu$ denote the Haar measure on $\SU(n)$ with volume one.
Let $W$ be any $\SU(n)$-module and let $W^{\SU(n)}$ be its subspace of invariants.
Consider the $\SU(n)$-module projection map $\pi: W\ra W^{\SU(n)}$. Then, the projection $\pi$ is explicitly realized as the integration: 
$$
\int_{\SU(n)} : W\ra W^{\SU(n)}\,, \,\,\,
v\mapsto 
\int_{g\in \SU(n)} g\cdot v \,d\mu  .$$

Assume further that $ W^{\SU(n)}$ is one dimensional. Take the unique (up to a scalar multiple)  nonzero element in the dual space $P\in {W^*}^{\SU(n)}$. Then, 
\be\label{int}\pi(v)\neq 0 \iff \langle P,  \int_{g\in \SU(n)} g\cdot v \,d\mu\rangle  \neq 0 \iff \langle P, v\rangle\neq 0.
\ene
In particular, take the vector space $V=\mathbb{C}^n$. Then,  $\End (V)$ is a $\GL(V)$ module under the left multiplication.    
For $\d=dn$ (for any $d\geq 0$),  there is a unique 
(up to scale) $\SL(V)$-invariant in $S^\d(\End (V))$, namely $\tdet_n^{\ot d}$ and there is none otherwise, see e.g., \cite[Thm. 5.6.7]{GoodmanWallach}.
We denote  $\tdet_n \in S^n(\End (V))$ by $\tdet_n^V$. Similarly, since $\End (V)$ is canonically isomorphic with the dual $\End (V)^*$,
we can think of $\tdet_n \in S^n(\End (V)^*)$. To distinguish, when
thinking  of $\tdet_n \in S^n(\End (V)^*)$, we denote it by 
$\tdet_n^{V^*}$.

These integrals have been extensively studied in the free probability and 
mathematical  physics literature, see, e.g., 
\cite{MR1959915,MR2217291}.
Despite this, the integrals that arose in our study do not appear to be known.

\subsection{The equivalences}
 
The following is the main result of this note.
\begin{theorem}\label{kconjequiv} Fix $n$ even. Let $V=\mathbb{C}^n$ and write $\End (V)=\Mat_n$, where $\Mat_n$ is the space of $n\times n$ matrices. Let $d\mu$ denote the Haar measure
on $\SU(n)$ and let $\SU(n)$ act on $\End (V)$  by left multiplication. Write $g^i_j$ for the coordinate
functions on $\End (V)$.  The following are equivalent:

(a) The Alon-Tarsi conjecture for $n$.

(b) Conjecture \ref{kumarconj} for $n$  with $d=n$.

(c) $\int_{g\in \SU(n)} g\cdot (\tperm_n^{V^*} )^n  d\mu\neq 0$.

(d)
$
\langle (\tperm_n^{V^*})^n,(\tdet_n^{V})^n\rangle \neq 0
$.

(e) $\int_{g \in \SU(n)} \bigl(\Pi_{1\leq i,j\leq n}g^i_j \bigr) d\mu\neq 0$.

(f) $\langle \Pi_{i,j}g^i_j, (\tdet_n^{V})^n\rangle\neq 0$.
\end{theorem}

The pairings in (d) and (f)
  may also be thought of as a pairing between homogeneous polynomials of degree $n^2$ and homogeneous differential operators
of order $n^2$.

Rectangular versions of these equivalences can be formulated as well.
\smallskip

\subsection{Motivation from geometric complexity theory}
In geometric complexity theory, see \cite{MS1,MS2,MR2861717,MR3093509}, one 
looks for modules that are   in the  ideal of  
the orbit closure $\ol{\GL_{n^2}\cdot \tdet_n}\subset S^n(\End (V)^*)$ 
of the determinant polynomial.  One approach to this search is
to find    modules  in $\Sym(S^n(\End (V)))$ that  do not occur in the coordinate
ring of the orbit $\GL_{n^2}\cdot \tdet_n$, which can in principle be determined
from representation theory, see \cite{MR2861717}.
The following observations are from  \cite{kumarcoordring} (where they
are explained in detail):
Since  $\Ch_n(\End (V)^*)\subset \ol{\GL_{n^2}\cdot \tdet_n}$,    any polynomial not in the
ideal of $\Ch_n(\End (V)^*)$ cannot be in the ideal of $\ol{\GL_{n^2}\cdot \tdet_n}$.
Thus, 
if $S_{(n^d)}(\BC^{n^2} )\not\subset I(\Ch_n(\BC^{n^2}))$, for all $1\leq d\leq n$,    then 
for any partition $\pi$  with at most $n$ parts, 
  the module
$S_{n\pi}\BC^{n^2}$ occurs at least once in $\BC[\ol{\GL_{n^2}\cdot \tdet_n}]$;
in particular, the symmetric Kronecker coefficient $s_{n\pi,  d^n, d^n}$ is non-vanishing (cf. 
\cite[$\S$ 6]{kumarcoordring}).

\section{Construction of the invariant}  
Let $V=\mathbb{C}^d$  and let  $\O\in \La dV^*$ be non-zero. Then, for any even $n$,  the 
one-dimensional module  $S_{(n^d)}V^*$ occurs with multiplicity one  in $S^d(S^nV^*)$ (cf. \cite[Proposition 4.3]{MR983608}). Write $\ol{\O}$ when 
considering
  $\O$ as a multi-linear form on $V$, and write $\O$ when using it as an element
of the dual space  $\La dV^*$ to $\La dV$.

\begin{proposition}\label{theinvaris} Let $n$ be even. 
Choosing the scale appropriately,  the unique (up to scale)
 polynomial $P\in S_{(n^d)}V^*\subset S^d(S^n V^*)$ evaluates on 
 $$x=(v^1_1\cdots v^1_n)(v^2_1\cdots v^2_n)\cdots (v^d_1\cdots v^d_n)\in S^d(S^nV), \,\,\,\text{for any}\,\, v^i_j\in V,
$$
 to give
\be\label{Pexpr}
\langle P,x\rangle=\sum_{\s_1\hd \s_d\in \FS_n}\ol{\O}(v^1_{\s_1(1)}\hd v^d_{\s_d(1)})
\cdots \ol{\O}(v^1_{\s_1(n)}\hd v^d_{\s_d(n)}).
\ene
\end{proposition}

\begin{proof} Let $\bar{P}\in  (V^*)^{\otimes nd}$ be defined by the identity \eqref{Pexpr} (with $P$ replaced by $\bar{P}$). It suffices to check that  

(i) $\bar{P}\in S^d(S^nV^*)$,

 (ii) $\bar{P}$ is $\SL(V)$ invariant, 
 and

(iii) $\bar{P}$ is not identically zero. 

Observe that (iii) follows from  the identity \eqref{Pexpr} by taking $v^i_j=e_i$ where $e_1\hd e_d$ is the standard  basis of $V$, and (ii) follows because $\SL(V)$ acts trivially
on $\Omega$.

\smallskip

To see (i), we show (ia) $\bar{P}\in S^d((V^*)^{\ot n})$ and (ib) $\bar{P}\in (S^nV^*)^{\ot d}$ to conclude.
To see (ia), it is sufficient to show that exchanging  two adjacent factors in parentheses in the expression of $x$
will not change \eqref{Pexpr}. Exchange   $v^1_j$ with $v^2_j$ in the expression for
$j=1\hd n$.
Then, each individual determinant will change sign, but there are an even number of determinants, so the right hand side
of \eqref{Pexpr} is unchanged. To see (ib),  it is sufficient to show the expression is unchanged if we swap
$v^1_1$ with $v^1_2$ in \eqref{Pexpr}. If we multiply by $n!$, we may assume $\s_1=\Id$, i.e., 
\begin{align*}&
\langle \bar{P},x\rangle=\\
&n!\sum_{\s_2\hd \s_d\in \FS_n}\ol{\O}(v^1_1,v^2_{\s_2(1)}\hd v^d_{\s_d(1)})
\ol{\O}(v^1_2,v^2_{\s_2(2)}\hd v^d_{\s_d(2)})
\cdots \ol{\O}(v^1_n,v^2_{\s_2(n)}\hd v^d_{\s_d(n)}).
\end{align*}
With the two elements $v^1_1$ and $v^1_2$ swapped, we get
\be \label{P2}
 n!\sum_{\s_2\hd \s_d\in \FS_n}\ol{\O}(v^1_2,v^2_{\s_2(1)}\hd v^d_{\s_d(1)})
\ol{\O}(v^1_1,v^2_{\s_2(2)}\hd v^d_{\s_d(2)})
\cdots \ol{\O}(v^1_n,v^2_{\s_2(n)}\hd v^d_{\s_d(n)}).
\ene
Now right compose each $\s_s$ in \eqref{P2} by the transposition $(1,2)$. The expressions become the same. 
\end{proof}

\section{Proof of the equivalences in Theorem \ref{kconjequiv}}
Let $V=\mathbb{C}^n$ with the standard basis $\{e_1, \dots, e_n\}$. The equivalences (c)$\Leftrightarrow$  (d) and (e) $\Leftrightarrow$ (f)  follow from the identity (\ref{int}) applied to the $\SU(n)$-module $W=S^{n^2}(\End (V)^*)$. (To prove the equivalence (e) $\Leftrightarrow$ (f), we have used the fact that any $\SU(n)$-invariant polynomial $Q\in S^{n^2}(\End (V)^*)$ is non-zero if and only if it does not vanish at $\Id \in \End (V)$.)

We now prove the other equivalences. We have two natural bases of 
$S^n(S^nV)_0$  to work with, the {\it monomial basis} consisting
of products  in the $e_j$ such that the $\fsl(V)$-weight of the expression
is zero,  and a     {\it weight basis}. To obtain a weight basis, 
first decompose $S^n(S^nV)$ into irreducible
$\GL(V)$-modules and then take a basis of the $\fsl(V)$-weight zero subspace of
each module. A  weight basis is the collection of the vectors in these spaces. (Observe that this basis is not unique.) 
The polynomial $P\in S_{(n^n)}(V^*)\subset S^n(S^n(V^*))$ will have a non-zero evaluation on 
$(e_1\cdots e_n)^n$ (equivalently, not be in the ideal of
the Chow variety)  if and only if, when expanding $P$ in the monomial basis 
obtained from the basis  $y_1\hd y_n$ dual  to $e_1\hd e_n$,
the coefficient of $(y_1\cdots y_n)^n$ is non-zero.

To see (a)$\Leftrightarrow$(b) (which was already shown in \cite{kumarcoordring}),
by the identity (\ref{Pexpr}) for $d=n$, 
\be \label{pairing}
\langle P,(e_1\cdots e_n)^n\rangle=\sum_{\s_1\hd \s_n\in \FS_n}\ol{\O}(e_{\s_1(1)}\hd e_{\s_n(1)})
\cdots \ol{\O}(e_{\s_1(n)}\hd e_{\s_n(n)}).
\ene
A term in the summation is non-zero if and only if the permutations $\s_1\hd \s_n$ give
rise to a Latin square by putting the values of $\sigma_i$ in the $i$-th row, and the contribution of the term is the column sign of the
square. This proves the equivalence
of (a) and (b) by Lemma \ref{hadlem}.

Now, 
\begin{align} \label{eqn5}
 \int_{g\in \SU(n)} g\cdot (e_1\cdots e_n)^n d\mu
&=
\int_{g\in \SU(n)}   ((g\cdot e_1)\cdots  (g\cdot e_n))^n d\mu \notag\\
&=\sum_{1\leq i^p_q\leq n}\,\int_{g\in \SU(n)}   (g_1^{i^1_1}\cdots  g_n^{i^1_n})
\cdots (g_1^{i^n_1}\cdots  g_n^{i^n_n}) (e_{i^1_1}\cdots e_{i^1_n})
\cdots  (e_{i^n_1}\cdots e_{i^n_n})
 d\mu \notag\\
 &=\sum_{1\leq i^p_q\leq n}\,[\int_{g\in \SU(n)}   (g_1^{i^1_1}\cdots  g_n^{i^1_n})
\cdots (g_1^{i^n_1}\cdots  g_n^{i^n_n})
d\mu]
 (e_{i^1_1}\cdots e_{i^1_n})
\cdots  (e_{i^n_1}\cdots e_{i^n_n}) \notag\\
&=[\int_{g\in \SU(n)} \sum_{\{ i^j_1\hd i^j_n\}=[n] \forall j}  (g_1^{i^1_1}\cdots  g_n^{i^1_n})
\cdots (g_1^{i^n_1}\cdots  g_n^{i^n_n})
d\mu]
 (e_{1}\cdots e_{n})^n
+ \notag\\
&
[\int_{g\in \SU(n)} \sum_{\s \in \FS_n: \{ i^j_1\hd i^j_n\} =\s(j)}  (g_1^{i^1_1}\cdots  g_n^{i^1_n})
\cdots (g_1^{i^n_1}\cdots  g_n^{i^n_n})
d\mu]
 (e_{1})^n
\cdots  ( e{_n})^n + x \notag\\
&=[\int_{g\in \SU(n)} (\tperm(g))^n
d\mu]
 (e_{1}\cdots e_{n})^n
+\notag\\
&
[\int_{g\in \SU(n)} (\Pi_{1\leq i,j\leq n}g^i_j)
d\mu]
 (e_{1})^n
\cdots  ( e{_n})^n +x,
\end{align}
where $x\in S^n(S^n(V))_0$ is in the span of the monomial basis not involving $(e_{1}\cdots e_{n})^n$ and $ (e_{1})^n
\cdots  ( e{_n})^n$.

Consider   the projection $
\int_{\SU(n)} : W\ra W^{\SU(n)}$ as in \S\ref{integral} for the $\SU(n)$-module  $S^n(S^n(V))$ and  for the unique $\SU(n)$-invariant $P^*\in S^n(S^n(V))$. It implies 
\begin{equation} \label{eqn6} \int_{g\in \SU(n)} g\cdot (e_1\cdots e_n)^n d\mu = \a P^*, \,\,\,\text{for some}\,\, \a \in \mathbb{C}.
\end{equation}

From this, together with the identity  (\ref{int}), we get the equivalence of (b) and the non-vanishing of $\int_{g\in \SU(n)} g\cdot (e_1\cdots e_n)^n d\mu$. Thus, the identity (\ref{eqn5}) shows that (e) implies (b). Further, assuming
 (b),
  the  identity (\ref{eqn6})    implies  $\int_{g\in \SU(n)} g\cdot (e_1\cdots e_n)^n d\mu$ is a non-zero multiple of $P^*$. But, by the proof of Proposition \ref{theinvaris}, $P^*$   contains the monomial 
$(e_1^n)\cdots (e_n^n)$ with non-zero coefficient. Thus,   identity (\ref{eqn5})
implies  (e). This shows the equivalence of (b) and (e). 

It is easy to see that (c) is equivalent to $\int_{g\in \SU(n)} (\tperm(g))^n
d\mu \neq 0$. Thus, by the identity (\ref{eqn5}) , (c) implies (b). Further, from the identity (\ref{Pexpr}) for $d=n$, it is easy to see that if (a) (equivalently (b)) is true, $P^*$ contains the monomial $ (e_{1}\cdots e_{n})^n$ with non-zero coefficient. Thus, from the identities (\ref{eqn5}) and (\ref{eqn6}), we get that (b) implies the non-vanishing of $\int_{g\in \SU(n)} (\tperm(g))^n
d\mu $, and hence (c). Thus (b) and (c) are equivalent. This proves the theorem. 
\qed

\bibliographystyle{amsplain}
 
\bibliography{Lmatrix.bib}

\end{document}

%% file: cortdefs.tex
\newtheoremstyle{custom}
  {3pt}
  {3pt}
  {\slshape}
  {}
  {\bfseries}
  {.}
  { }
   {}
\theoremstyle{custom}
\newtheorem{theorem}{Theorem}[section]
\newtheorem{proposition}[theorem]{Proposition}

\newtheorem{proposition/definition}[theorem]{Proposition/Definition}
\newtheorem{lemma}[theorem]{Lemma}

\newtheorem{conjecture}[theorem]{Conjecture}

\theoremstyle{definition}

\theoremstyle{remark}
\newtheorem{remark}[theorem]{Remark}

\newtheoremstyle{exercise}
  {3pt}
  {6pt}
  {}
  {}
  {\bfseries}
  {:}
  { }
   {}
\theoremstyle{exercise}
\newtheorem{exercise}[theorem]{Exercise}
\newtheoremstyle{exercises}
  {3pt}
  {6pt}
  {}
  {}
  {\bfseries}
  {:}
  {\newline}
   {}

\theoremstyle{exercise}
\newtheorem{exercises}[theorem]{Exercises}

\input epsf
\def\boxit#1{\vbox{\hrule height1pt\hbox{\vrule width1pt\kern3pt
  \vbox{\kern3pt#1\kern3pt}\kern3pt\vrule width1pt}\hrule height1pt}}

\def\BC{\mathbb C}

\def\tdim{{\rm dim}}

\def\hd{,...,}

\def\cW{{\mathcal W}}

\def\11{\mathbf 1}

\def\fsl{{\mathfrak {sl}}}

\def\a{\alpha}

\def\O{\Omega}

\def\s{\sigma}

\def\d{\delta}

\def\ot{{\mathord{ \otimes } }}

\def\otc{{\mathord{\otimes\cdots\otimes}\;}}

\def\ra{{\mathord{\;\rightarrow\;}}}

\def\dim{{\rm dim}\;}
\def\La#1{\Lambda^{#1}}

\def\frak{\mathfrak}
\def\fsl{\frak s\frak l}

\def\s{\sigma}

\def\a{\alpha}

\def\FS{\mathfrak  S}

\def\ol{\overline}

\def\BC{\mathbb  C}

\def\hd{, \hdots ,}

\def\La#1{\Lambda^{#1}}

\def\ra{\rightarrow}

\def\tdet{\operatorname{det}}

\def\tperm{\operatorname{perm}}

\def\tdim{\operatorname{dim}}
\def\tker{\operatorname{ker}}

\def\ctimes{\times \cdots\times}

\def\be{\begin{equation}}
\def\ene{\end{equation}}

\newcommand{\End}{\operatorname{End}}

\newcommand{\Id}{\operatorname{Id}}

